\documentclass[12pt,a4paper,reqno]{amsart}
\usepackage{amsmath}
\usepackage{amsfonts}
\usepackage{amssymb}
\usepackage{amsthm}
\usepackage{enumerate}
\usepackage{centernot}
\usepackage{fullpage}
\usepackage{mathrsfs}
\usepackage{graphicx}
\usepackage{tikz-cd}
\usepackage{color}

\newcommand{\upperRomannumeral}[1]{\uppercase\expandafter{\romannumeral#1}}

\theoremstyle{plain}

  \newtheorem{proposition}[]{Proposition}
  \newtheorem{lemma}[]{Lemma}
  \newtheorem{theorem}[]{Theorem}
  \newtheorem{corollary}[]{Corollary}
  
  \newtheorem{remark}[]{Remark}

\title[Mass spectrum in the Ising model]{The Gaussian process for particle masses in the near-critical Ising model}
\author{Federico Camia}
\address{Division of Science, NYU Abu Dhabi, Saadiyat Island, Abu Dhabi, UAE \& Courant Institute of Mathematical Sciences, New York University, 251 Mercer st, New York, NY 10012, USA \& Department of Mathematics, Faculty of Science, Vrije Universiteit Amsterdam, De Boelelaan 1111, 1081 HV Amsterdam, The Netherlands.}
\email{federico.camia@nyu.edu}
\author{Jianping Jiang}
\address{Yanqi Lake Beijing Institute of Mathematical Sciences and Applications, Building 11, Yanqi Island, Yanqi Lake West Road, Beijing 101408, China.}
\email{jianpingjiang11@gmail.com}
\author{Charles M. Newman}
\address{Courant Institute of Mathematical Sciences, New York University, 251 Mercer st, New York, NY 10012, USA \& NYU-ECNU Institute of Mathematical ciences at NYU Shanghai, 3663 Zhongshan Road North, Shanghai 200062, China.}
\email{newman@cims.nyu.edu}

\begin{document}
\begin{abstract}
We review the construction of a stationary Gaussian process $X(t)$ starting from the near-critical
continuum scaling limit $\Phi^h$ of the Ising magnetization and its relation to the mass spectrum of the relativistic quantum field theory
associated to $\Phi^h$. Then for 
the near-critical Ising model on $a \mathbb{Z}^2$ with external field $a^{15/8} h$,
we study the renormalized magnetization along a vertical line (with horizontal coordinate 
approximately~$t$) and prove that the limit as $a\downarrow 0$ is the 
same Gaussian process $X(t)$.  We also explore the possible extension 
of this approach to dimensions $d > 2$.
\end{abstract}
\maketitle

\section{Introduction}
\subsection{Overview}
Consider the Ising model at inverse critical temperature $\beta_c$ on $a\mathbb{Z}^2$ with $a>0$ and external field $a^{15/8}h\geq0$. Let $\{\sigma_x:x\in a\mathbb{Z}^2\}$ denote the basic spin random 
variables of the Ising model, and $\Phi^{a,h}$ denote the magnetization field
\begin{equation}
	\Phi^{a,h}:=a^{15/8}\sum_{x\in a\mathbb{Z}^2}\sigma_x\delta_x \, , 
\end{equation}
where $\delta_x$ is a unit Dirac point measure at $x$. In Theorem 1.2 of \cite{CGN15} and Theorem 1.4 of~\cite{CGN16}, it was proved that
\begin{equation}\label{eq:Phidef}
	\Phi^{a,h}\Rightarrow \Phi^h
\end{equation}
where $\Rightarrow$ denotes convergence in distribution, and $\Phi^h$ with 
$h\geq 0$ is a generalized random field on $\mathbb{R}^2$; see 
\cite{CJN20d} for a review. Euclidean random fields such as $\Phi^h$ on Euclidean 
``space-time" $\mathbb{R}^d:=\{x=(x_0,y_1,\dots,y_{d-1})\}$ are related to quantum fields on relativistic space-time, $\{(t,y_1,\dots,y_{d-1})\}$, essentially by replacing $x_0$ with a complex variable and analytically continuing from the  purely real $x_0$ to pure imaginary $(-it)$---see \cite{OS73, OS75}, Chapter 3 of \cite{GJ87} and \cite{MM97} for background. It was predicted in \cite{Zam89a, Zam89b} that the relativistic quantum field theory associated to the Euclidean field $\Phi^h$ with $h>0$ should have remarkable properties including the existence of eight distinct types of particles, with relations between the masses of those particles and 
the Lie algebra $E_8$ \cite{Del04, BG11, MM12}---see also 
\cite{CTW+10} (respectively, \cite{CH00}) for experimental (respectively, numerical) studies.

In \cite{CJN20a, CJN20b}, exponential decay of truncated correlations in $\Phi^h$ with $h>0$ was proved; this roughly says that in the relativistic quantum field theory associated to $\Phi^h$ with $h>0$, there is at least one particle with strictly positive mass and no smaller mass particles. In~\cite{CJN20c}, the authors took the limit of $\Phi^h(t,y)$ as the spatial coordinate $y$ scales to infinity with $t$ fixed and proved that it is a stationary Gaussian process $X(t)$ whose covariance function $K(t)$ should provide a useful tool for analyzing particle masses in the associated quantum field theory.

\subsection{Why are $X$ and $K$ of interest?} 
$K$ should capture some important information about particle masses of the quantum field 
theory associated with $\Phi^h$. In \cite{CJN20c}, based on \cite{Del04, Zam89a, Zam89b} 
and \eqref{eq:Kint} below, it was conjectured that there exist $m_1,m_2,m_3\in(0,\infty)$ 
and $B_1,B_2,B_3\in(0,\infty)$ such that, for large $t$,
\begin{equation}
	K(t)=B_1\exp(-m_1|t|)+B_2\exp(-m_2|t|)+B_3\exp(-m_3|t|)+O(\exp(-2m_1|t|)),
\end{equation}
with $m_1<m_2<m_3<2m_1$; the mass $m_1$ should be the same as in \eqref{eq:Kint}, 
and $m_2/m_1$, $m_3/m_1$ should take the predicted values, as in (1.8) of \cite{Zam89a}.

The exponential decay result in \cite{CJN20a, CJN20b} only shows that in the relativistic 
quantum field theory associated to $\Phi^h$, there is a mass gap---i.e., no particles with 
masses in $[0,m_1)$. A natural question is then whether the mass spectral measure 
$\rho$ (defined in \eqref{eq:tilderhodef} below) has an atom with strictly positive weight at 
$m_1$. As proved in Appendix B of \cite{CJN20c}, this would follow from Ornstein-Zernike 
behavior of the covariance function $H(t,y)$ of $\Phi^h$. We state that result as 
a proposition here.
\begin{proposition}[Propositions A and B of \cite{CJN20c}]
	Suppose that there exist a constant $C_1\in(0,\infty))$ such that
	\begin{equation}\label{eq:HOZ}
		\lim_{t\rightarrow\infty}\frac{H(t,0)}{t^{-1/2}\exp(-m_1t)}=C_1=\rho\left(\{m_1\}\right)\sqrt{m_1/(2\pi)}.
	\end{equation}
Then we have 
\begin{equation}
	\lim_{t\rightarrow\infty}\frac{K(t)}{\exp(-m_1t)}=C_1\sqrt{2\pi/m_1}=\rho\left(\{m_1\}\right).
\end{equation}
\end{proposition}

Recently, a new proof of exponential decay (for truncated correlations in $\Phi^h$) based on 
the random current representation of the Ising model was given in \cite{KR21}. We believe that it is possible to combine the methods in \cite{Ott20b} and \cite{KR21} to give a rigorous proof of \eqref{eq:HOZ}. Then the next step would be to show that the mass spectrum has an upper gap $(m_1,m_1+\epsilon)$ for some $\epsilon>0$.

In this paper, we first review in Section \ref{sec:rev} the main results of \cite{CJN20c}. Then in 
Section \ref{sec:new}, we state some new results which are closely related to those of 
Section \ref{sec:rev}. 
In Section~\ref{subsec:2dim}, the main result is that the same Gaussian process
$X(t)$ can be obtained directly from the near-critical lattice Ising model on $a\mathbb{Z}^2$ by
appropriate scalings of the vertical and horizontal coordinates without use of the continuum field 
$\Phi^h$. Then in Section~\ref{subsec:alldim}, we
explore the possible extension of 
this approach to dimensions $d>2$. In Section \ref{sec:proofs}, we prove the results 
stated in Section \ref{sec:new}.

\section{The Gaussian process  from $\Phi^h$: a review}\label{sec:rev}
\subsection{Construction of the Gaussian process}
Let $H(t,y)$ be the covariance function of $\Phi^h$. Roughly speaking,
\begin{equation}
	H(t,y):=\text{Cov}\left(\Phi^h(t_0,y_0),\Phi^h(t_0+t,y_0+y)\right) \text{ for any }(t_0,y_0)\in\mathbb{R}^2.
\end{equation}
The existence of $H$ follows from Proposition 6.1.4 of \cite{GJ87}; this is because 
$\mathbb{E}\left(\exp(iz\Phi^h(f))\right)$ is analytic in $z$, which can be proved,
for example, by arguments based on the Lee-Yang theorem \cite{LY52} and the GHS 
inequality~\cite{GHS70}. For each fixed $L>0$, we define a collection of random variables 
$\{X_L(s):s\in\mathbb{R}\}$ by
\begin{equation}
	X_L(s):=\frac{\Phi^h\left(1_{[-L,L]}(y)\delta_s(t)\right)-\mathbb{E}\Phi^h\left(1_{[-L,L]}(y)\delta_s(t)\right)}{\sqrt{2L}},
\end{equation}
where $1_{[-L,L]}(y)\delta_s(t)$ is the product of an interval indicator function in $y$ and a delta function in $t$, and $\mathbb{E}$ is the expectation with respect to the random field $\Phi^h$. Formally,
\begin{equation}\label{eq:Phioff}
	\Phi^h(f):=\int_{\mathbb{R}^2}\Phi^h(x)f(x)dx.
	\end{equation}
In \cite{CGN16}, it was shown that \eqref{eq:Phioff} is well-defined for any $f$ in the dual of the Sobolev 
space $\mathcal{H}^{-3}(\mathbb{R}^2)$. This was later generalized in \cite{FM17} to any $f$ 
in the dual of the Besov space $\mathcal{B}_{p,q}^{-1/8-\epsilon,\text{loc}}(\mathbb{R}^2)$ where 
$\epsilon>0$ and $p,q\in[1,\infty]$. Since the test function $1_{[-L,L]}(y)\delta_s(t)$ is in 
neither of those two spaces, we instead refer to Lemma A in Appdendix A of \cite{CJN20c} for a 
justification of such a paring; the idea is to approximate the delta fnction with smooth 
functions. Let $\{X(s):s\in\mathbb{R}\}$ be a mean zero stationary Gaussian process with 
covariance function
\begin{equation}\label{eq:Kdef}
	\text{Cov}(X(s),X(t))=K(t-s):=\int_{-\infty}^{\infty}H(t-s,y)dy \text{ for any } s, t\in\mathbb{R}.
\end{equation}

The following was proved in \cite{CJN20c}. 
\begin{theorem}[Theorem 1 of \cite{CJN20c}]\label{thm:contGaus}
	Fix $h>0$. For any $n\in\mathbb{N}$ and distinct $s_1,\dots,s_n\in\mathbb{R}$, 
	\begin{equation}
		\left(X_L(s_1),\dots,X_L(s_n)\right)\Rightarrow \left(X(s_1),\dots,X(s_n)\right) \text{ as }L\rightarrow\infty,
	\end{equation}
where $\Rightarrow$ denotes convergence in distribution.
\end{theorem}

$H$ is really a function of the radial variable $\sqrt{t^2+y^2}$. Note that $H$ actually 
depends on $h$; we only distinguish whether $h=0$ or not since all results in this paper are 
insensitive to the exact value $h>0$. When $h=0$, we always write $H^0$ for the 
covariance function of $\Phi^0$. By Proposition \ref{prop:Cov} below and Wu's 
result~\cite{Wu66, MW73} (see also Remark 1.4 of \cite{CHI15} and Theorem 3.1 of \cite{Che18}), 
we have
\begin{equation}
	H^0(t,y)=C_2(t^2+y^2)^{-1/8}, \text{ for any } (t,y)\in\mathbb{R}^2 \text{ with } (t,y) \neq (0,0),
\end{equation}
where $C_2\in(0,\infty)$ is a universal constant. The following small distance/time behavior of $H$ and $K$ was also proved  in \cite{CJN20c}.
\begin{theorem} [Theorem 2 of \cite{CJN20c}]\label{thm:HKsmalldist}
	Fix $h>0$.
	\begin{align}
		&\lim_{\lambda\downarrow 0} \lambda^{1/4} H(0,\lambda)=H^0(0,1)=C_2.\\
		&\lim_{\epsilon\downarrow 0}\frac{K(0)-K(\epsilon)}{\epsilon^{3/4}}=2\int_0^{\infty}\left[H^0(0,y)-H^0(1,y)\right]dy\in(0,\infty).
	\end{align}
\end{theorem}

\subsection{Relation to quantum field theory}
Since $\Phi^h$ is a Euclidean field satisfying the Osterwalder-Schrader axioms, by the 
K\"all\'en-Lehmann spectral formula (see Theorem~6.2.4 of \cite{GJ87}), we have
\begin{equation}
	H(s,y)=\int_0^{\infty}\left(\int_{-\infty}^{\infty}\int_0^{\infty} \exp(ipy)\exp(-E|s|)\delta(m^2 +p^2-E^2)dE dp\right)d\tilde{\rho}(m),
\end{equation}
where $\tilde{\rho}$ is a mass spectral measure of the relativistic quantum field theory obtained from $\Phi^h$ via the Osterwalder-Schrader reconstruction theorem \cite{OS73, OS75}. Here for fixed $p$ and $m$,
\begin{equation}
	\delta(m^2+p^2-E^2)=\frac{\delta(\sqrt{m^2+p^2}+E)+\delta(\sqrt{m^2+p^2}-E)}{2\sqrt{m^2+p^2}}.
\end{equation}
Then an easy computation (see \cite{CJN20c} for the details) gives
\begin{equation}
	K(s)=\pi\int_0^{\infty}\frac{\exp(-|s|m)}{m}d\tilde{\rho}(m), \text{ for any }s\in\mathbb{R}.
\end{equation}
By Theorem 1.4 and Remark 1.7 of \cite{CJN20a}, the support of $\tilde{\rho}$ is in $[m_1,\infty)$ for some $m_1>0$. If we define a new measure $\rho$ by the Radon-Nikodym derivative
\begin{equation}\label{eq:tilderhodef}
	\frac{d\rho}{d\tilde{\rho}}(m)=\frac{\pi}{m} \,  ,
\end{equation}
then we have the following.
\begin{proposition}[Theorem 1 and Corollary 1 of \cite{CJN20c}]
	There exists $m_1\in(0,\infty)$ such that
	\begin{equation}\label{eq:Kint}
		K(s)=\int_{m_1}^{\infty}\exp(-m|s|)d\rho(m).
	\end{equation}
	Moreover, $\rho$ is a finite measure, but its first moment is infinite.
\end{proposition}

\section{A Gaussian process from the discrete Ising model}\label{sec:new}
\subsection{Two-dimensional results}\label{subsec:2dim}
Theorem \ref{thm:contGaus}  says that the limit of the centered $\Phi^h$ as the spatial 
coordinate $y$ scales to infinity with 
the Euclidean time coordinate $t$ fixed is a mean zero stationary Gaussian process 
$X(t)$ with covariance function $K(t)$. We prove in Section~\ref{sec:proofs} below
that this Gaussian process can also be 
obtained directly from the near-critical Ising model on $a\mathbb{Z}^2$.

Denote by $P_h^a$ the infinite volume Ising measure at the inverse critical temperature $\beta_c$ on 
$a\mathbb{Z}^2$ with external field $a^{15/8}h> 0$. Let $\langle\cdot\rangle_{a,h}$ denote 
expectation with respect to $P_h^a$. For $s\in \mathbb{R}$, let $s_a$ denote a point in 
$a\mathbb{Z}$ that is closest to $s$. For $L>0$ and $s\in\mathbb{R}$, we define
\begin{equation}\label{eq:DefX_L}
X_{L}(s):=\frac{a^{7/8}\sum_{k\in a\mathbb{Z}\cap [-L,L]}\left[\sigma_{(s_a,k)}-\langle\sigma_{(s_a,k)}\rangle_{a,h}\right]}{\sqrt{2L}}.
\end{equation}

Our main result for $d=2$ is the following.
\begin{theorem}\label{thm:main}
Suppose $L(a)>0$ is a function of $a$ satisfying $L(a)\rightarrow\infty$ as $a\downarrow 0$.
Then for any $n\in\mathbb{N}$ and distinct $s_1,\dots,s_n\in\mathbb{R}$, we have
\begin{equation}\label{eq:thmmain}
\left(X_{L(a)}(s_{1,a}),\dots,X_{L(a)}(s_{n,a})\right)\Rightarrow (X(s_1),\dots,X(s_n)) \text{ as }a\downarrow 0,
\end{equation}
where $s_{j,a}$ for each $1\leq j\leq n$ is a point in $a\mathbb{Z}$ that is closest to $s_j$.
\end{theorem}

The relation between Theorems \ref{thm:contGaus} and \ref{thm:main}, and the results in \cite{CGN16} can be summarized in the following diagram:
\[ \begin{tikzcd}
    \{\sigma_x: x\in a\mathbb{Z}^2\} \arrow{r}{a\downarrow0} \arrow[swap]{rd} &\Phi^h(t,y) \arrow{d}{L\uparrow\infty}\\
                                                                                                                        &X(t),
    \end{tikzcd}
\]
where the right and down arrows represent results in \cite{CGN16} and Theorem \ref{thm:contGaus} respectively, and the diagonal arrow represents Theorem \ref{thm:main}.

\subsection{Results and discussions for general dimension}\label{subsec:alldim}
For general $d\geq2$, we try to derive the desired Gaussian process (denoted by $\hat{X}$) in two steps. In the first step, we construct a Gaussian process $\hat{X}^a(t)$ from the critical Ising model on $a\mathbb{Z}^d$ with fixed magnetic field $\tilde{h}$; in the second step, we choose $\tilde{h}$ as an appropriate function of $a$ and obtain $\hat{X}$ from $\hat{X}^a(t)$ by sending $a\downarrow0$.

On $a\mathbb{Z}^d$, we write $x\in a\mathbb{Z}^d$ as $x=(t,\vec{y})$ where $t\in a\mathbb{Z}$ and $\vec{y}\in a\mathbb{Z}^{d-1}$. For $L>0$ and $s\in \mathbb{R}$, we define
\begin{equation}\label{eq:DefhatX_L}
\hat{X}^a_{L}(s):=\frac{\sum_{\vec{y}\in a\mathbb{Z}^{d-1}\cap [-L,L]^{d-1}}\left[\sigma_{(s_a,\vec{y})}-\langle\sigma_{(s_a,\vec{y})}\rangle\right]}{\sqrt{\text{Var}\left(\sum_{\vec{y}\in a\mathbb{Z}^{d-1}\cap [-L,L]^{d-1}}\sigma_{(s_a,\vec{y})}\right)}},
\end{equation}
where $\langle\cdot\rangle$ denotes the expectation for the critical Ising model on 
$a\mathbb{Z}^d$ with fixed magnetic field $\tilde{h}$; for general $d\geq2$ we do not add 
any subscripts to brackets since the corresponding probability measure should be clear from the context.
Note that $\hat{X}^a_{L}(s)$ is standardized with mean $0$ and variance $1$. For 
$z, w\in a\mathbb{Z}^d$, let $\langle \sigma_z;\sigma_w\rangle$ denote the truncated two-point function, i.e.,
\begin{equation}
\langle \sigma_z;\sigma_w\rangle=\langle \sigma_z\sigma_w\rangle-\langle \sigma_z\rangle\langle \sigma_w\rangle.
\end{equation}
In the first step, we prove
\begin{theorem}\label{thm:gend}
Consider the critical Ising model on $a\mathbb{Z}^d$ with fixed magnetic field $\tilde{h}>0$. Then for any $n\in\mathbb{N}$ and distinct $s_1,\dots,s_n\in a\mathbb{Z}$, we have
\begin{equation}
\left(\hat{X}^a_L(s_1),\dots,\hat{X}^a_L(s_n)\right)\Rightarrow (\hat{X}^a(s_1),\dots,\hat{X}^a(s_n)) \text{ as }L\uparrow \infty,
\end{equation}
where $\{\hat{X}^a(s), s\in a\mathbb{Z}\}$ is a mean zero stationary Gaussian process with covariance function
\begin{equation}
\emph{Cov}(\hat{X}^a(s),\hat{X}^a(t))=\hat{K}^a(t-s)
:=\frac{\sum_{\vec{y}\in a\mathbb{Z}^{d-1}}\langle \sigma_{(0,\vec{0})};\sigma_{(t-s,\vec{y})}\rangle}{\sum_{\vec{y}\in a\mathbb{Z}^{d-1}}\langle \sigma_{(0,\vec{0})};\sigma_{(0,\vec{y})}\rangle} \emph{ for }s,t \in a\mathbb{Z}.
\end{equation}
\end{theorem}
\begin{remark}
$\hat{K}^a(s)$ is non-increasing as a function of~$|s|$. This follows from the monotonicity of $\langle \sigma_{(0,\vec{0})}\sigma_{(t,\vec{y})}\rangle$ (and hence also $\langle \sigma_{(0,\vec{0})};\sigma_{(t,\vec{y})}\rangle$) in $|t|$ (see \cite{Sch77} and~\cite{MMS77}).
\end{remark}

We recall that the correlation function for the critical Ising model on $\mathbb{Z}^d$ with 
$\tilde{h}=0$ is expected to scale in the following way
\begin{equation}
\langle\sigma_{\vec{0}}\sigma_{\vec{x}}\rangle\approx |\vec{x}|^{-d+2-\eta} \text{ for }\vec{0}, \vec{x}\in \mathbb{Z}^d \text{ and large } |\vec{x}|,
\end{equation}
where $|\vec{x}|:=\|\vec{x}\|_2$ denotes the Euclidean distance and $\eta\geq 0$. 
It is known that $\eta=1/4$ when $d=2$ \cite{Wu66, MW73} and the conjecture for 
$d\geq 4$ is $\eta=0$.  If we take $\tilde{h}=a^{(d+2-\eta)/2}h$ for some $h>0$, 
we conjecture that the following limit exists
\begin{equation}\label{eq:conj}
\hat{K}(s):=\lim_{a\downarrow 0} \hat{K}^a(s_a)=\lim_{a\downarrow 0}\frac{\sum_{\vec{y}\in a\mathbb{Z}^{d-1}}\langle\sigma_{(0,\vec{0})};\sigma_{(s_a,\vec{y})}\rangle}{\sum_{\vec{y}\in a\mathbb{Z}^{d-1}}\langle\sigma_{(0,\vec{0})};\sigma_{(0,\vec{y})}\rangle}, \text{ for }s\in\mathbb{R}.
\end{equation}
Now it is easy to prove (by showing the convergence of the corresponding characteristic functions)
\begin{proposition}
Under the assumption that the limit in \eqref{eq:conj} exists, we have
\begin{equation}
(\hat{X}^a(s_1),\dots,\hat{X}^a(s_n))\Rightarrow (\hat{X}(s_1),\dots,\hat{X}(s_n)) \text{ as }a\downarrow0,
\end{equation}
where $\{\hat{X}(s):s\in\mathbb{R}\}$ is a mean zero stationary Gaussian process with covariance function
\begin{equation}
\emph{Cov}(\hat{X}(s), \hat{X}(t))=\hat{K}(t-s).
\end{equation}
\end{proposition}

\begin{remark}\label{rem:comm}
From \eqref{eq:conj} and Remark \ref{rem:Cov} below, one can see that $\hat{K}(s)=K(s)/K(0)$ when $d=2$. Therefore, $\hat{X}(s)\overset{d}{=}X(s)/\sqrt{K(0)}$ when $d=2$.
\end{remark}

As mentioned in \cite{CJN20c}, for $d=3$, the covariance function $\hat{K}(t)$ for small $t$ 
would be nondifferentiable at $t=0$, like in Theorem \ref{thm:HKsmalldist} but with 
$\hat{K}(0)-\hat{K}(\epsilon)$ behaving like $\epsilon^{1-\eta}$ as $\epsilon\downarrow 0$,
rather than $\epsilon^{3/4}$, with $\eta$ the corresponding critical exponent for $d=3$. For 
$d>4$, $\hat{K}$ would be differentiable while for $d=4$, there is the possibility of logarithmic behavior.

The diagram right after Theorem \ref{thm:main} and Remark \ref{rem:comm} imply that, when $d=2$, 
the two limits ``$L\uparrow\infty$" and ``$a\downarrow 0$" commute. This should be 
contrasted with a result and a conjecture for the high dimensional Ising model in \cite{CJN21}. 
In Remark 3 of \cite{CJN21}, it was proved that for large $d$, the limit of the near-critical 
magnetization field on $\Lambda_L^a:=[-L,L]^d\cap a\mathbb{Z}^d$, $\Phi^h_{\Lambda_L}$, 
converges (after subtracting its mean) as $L\uparrow\infty$ in distribution to a \textit{massless} Gaussian free 
field on $\mathbb{R}^d$. But Conjecture 1 of \cite{CJN21} says that the near-critical 
magnetization field on $a\mathbb{Z}^d$ converges (after subtracting its mean) as $a\downarrow 0$ in distribution 
to a \textit{massive} Gaussian free field on $\mathbb{R}^d$. So in that case, the two limits 
``$L\uparrow\infty$" and ``$a\downarrow 0$" should not commute.

\section{Proof of results in Section \ref{sec:new}}\label{sec:proofs}
The main ingredients for the proof of Theorem \ref{thm:main} are the convergence of the covariance $\text{Cov}(X_{L(a)}(s), X_{L(a)}(t))$ and an inequality for FKG systems from \cite{New80}.  Since $H$ is a function only of the radial variable, we define
\begin{equation}
\hat{H}(\sqrt{t^2+y^2}):=H(t,y), \text{ for any } (t,y)\in\mathbb{R}^2.
\end{equation}
Recall that $s_a$ is a point in $a\mathbb{Z}$ that is closest to $s$. For $z\in\mathbb{R}^2$, let $z_a$ be a point in $a\mathbb{Z}^2$ that is closest to $z$.
\begin{proposition}\label{prop:Cov}
Fix $h\geq 0$. For any $L\in(0,\infty)$ and any $s,t\in\mathbb{R}$, we have
\begin{align}
&\lim_{a\downarrow 0} a^{3/4}\sum_{k\in a\mathbb{Z}\cap [-L,L]}\langle \sigma_{(s_a,0)}; \sigma_{(t_a,k)} \rangle_{a,h} =\int_{-L}^L H(t-s,y)dy,\label{eq:Covlimit}\\
&\lim_{a\downarrow 0} a^{-1/4}\langle \sigma_{z_a};\sigma_{w_a}\rangle_{a,h} =\hat{H}(|z-w|), \text{ for all } z\neq w\in\mathbb{R}^2.\label{eq:ttlimit}
\end{align}
\end{proposition}
\begin{remark}
The limit \eqref{eq:ttlimit} generalizes a classical result by Wu, which corresponds to $h=0$ in \eqref{eq:ttlimit}. See \cite{Wu66, MW73} (also \cite{CHI15} and Theorem~3.1 of~\cite{Che18}) for Wu's result.
\end{remark}
We recall two inequalities which will be important for the proof of Proposition~\ref{prop:Cov}. 
The first one is the SMM (for Schrader, Messager and Miracle-Sole) inequality \cite{Sch77, MMS77}: in a region $\Lambda$ with reflection symmetry, the correlation 
$\langle\prod_{x\in A}\sigma_x\prod_{x\in B}\sigma_x\rangle_{\Lambda,a,h}$ with $A$ and $B$ 
on the same side of a reflection plane can only decrease when $B$ is replaced 
by its reflected image $\bar{B}$, i.e.,
\begin{equation}
	\left\langle\prod_{x\in A}\sigma_x\prod_{x\in B}\sigma_x\right\rangle_{\Lambda,a,h}\geq \left\langle\prod_{x\in A}\sigma_x\prod_{x\in \bar{B}}\sigma_x\right\rangle_{\Lambda,a,h}.
\end{equation}
In the infinite volume limit on $a\mathbb{Z}^2$, this inequality holds for reflections with respect to
\begin{enumerate}[(a)]
	\item lines parallel to the principal axes and passing through 
	points in $(\frac{1}{2}a\mathbb{Z})\times(\frac{1}{2}a\mathbb{Z})$ --- in 
	particular for any $z=(z_1,z_2), w=(w_1,w_2)\in a\mathbb{Z}^2$,
	\begin{equation}\label{eq:Sch}
		\langle \sigma_0\sigma_z\rangle_{a,h}\leq \langle \sigma_0\sigma_w\rangle_{a,h}  \text{ if } z_1=w_1 \text{ and } |z_2|\geq |w_2|;
	\end{equation}
   \item ``diagonal" lines, i.e., lines with slope $\pm 1$ and passing through points in $a\mathbb{Z}^2$.
\end{enumerate}
The SMM inequality also holds for infinite-volume truncated two-point functions 
since the one-point function is constant. The second inequality is the GKS inequality (see 
Corollary~1 of \cite{Gri67b} and also \cite{KS68}), which says that
\begin{equation}\label{eq:Gri}
	\langle \sigma_z;\sigma_w\rangle_{a,h} \geq 0 \text{ for any }z, w\in a\mathbb{Z}^2.
\end{equation}

We will also use the following lemma about convergence of moments.
\begin{lemma}\label{lem:momconv}
	Fix $h\geq0$. Suppose $f\in C_c^{\infty}(\mathbb{R}^2)$, the space of $C^{\infty}$ functions on $\mathbb{R}^2$ whose support is compact. Then we have
	\begin{equation}\label{eq:momconv}
		\lim_{a\downarrow 0}\langle \Phi^{a,h}(f)\rangle_{a,h} =\mathbb{E}\Phi^h(f).
	\end{equation}
Moreover, for any $f,g\in C_c^{\infty}(\mathbb{R}^2)$, we have
\begin{equation}\label{eq:2momconv}
	\lim_{a\downarrow 0}\left\langle\Phi^{a,h}(f)\Phi^{a,h}(g)\right\rangle_{a,h}=\mathbb{E}\left( \Phi^h(f)\Phi^f(g)\right).
\end{equation}
\end{lemma}
\begin{remark}
	Lemma \ref{lem:momconv} actually holds for any bounded $f,g$ in the dual of the Sobolev 
	space $\mathcal{H}^{-3}(\mathbb{R}^2)$. In our applications, $f$ and $g$ 
	will be indicator functions of some bounded regions in~$\mathbb{R}^2$. 
\end{remark}
\begin{proof}[Proof of Lemma \ref{lem:momconv}]
	By Theorem 1.2 of \cite{CGN15} and Theorem 1.4 of \cite{CGN16}, we have
	\begin{equation}\label{eq:Phifconv}
		\Phi^{a,h}(f)\Rightarrow \Phi^h(f).
	\end{equation}
So \eqref{eq:momconv} follows if we can show that
\begin{equation}\label{eq:mombd}
	\left\langle \exp\left(t\Phi^{a,h}(f)\right)\right\rangle_{a,h} \text{ is bounded for any }t\in\mathbb{R}
\end{equation}
(uniformly as $a \downarrow 0$). 
We may choose a bounded $\Lambda$ such that the support of $f$ is contained in $\Lambda$. Then the GKS inequalities \cite{Gri67, KS68} imply that
\begin{equation}\label{eq:mombd1}
	\left\langle \exp\left(t\Phi^{a,h}(f)\right)\right\rangle_{a,h}\leq \left\langle \exp\left(\Phi^{a,0}\left((|t|\|f\|_{\infty}+h)1_{\Lambda}\right)\right)\right\rangle_{\Lambda,a,0}^+,
\end{equation}
where $\langle\cdot\rangle_{\Lambda,a,0}^+$ is the expectation of the critical Ising model on $\Lambda\cap a\mathbb{Z}^2$ with plus boundary conditions. The GHS inequality \cite{GHS70} gives that, for any $M\in\mathbb{R}$,
\begin{equation}\label{eq:mombd2}
	\left\langle \exp\left(\Phi^{a,0}(M1_{\Lambda})\right)\right\rangle_{\Lambda,a,0}^{+}\leq\exp\left[\langle\Phi^{a,0}(M1_{\Lambda})\rangle_{\Lambda,a,0}^++\frac{1}{2}\text{Var}_{\Lambda,a,0}^+(\Phi^{a,0}(M1_{\Lambda}))\right] \, .
\end{equation}
By Proposition B.1 in Appendix B.1 of \cite{CGN15}, which bounds the mean and 
variance in~\eqref{eq:mombd2},  
we see that \eqref{eq:mombd1} and \eqref{eq:mombd2} imply \eqref{eq:mombd}. 
For the proof of \eqref{eq:2momconv}, we note that
\begin{equation}
	\Phi^{a,h}(f+g)\Rightarrow \Phi^h(f+g) \text{ as }a\downarrow 0,
\end{equation}
and \eqref{eq:Phifconv}  and \eqref{eq:mombd} imply that
\begin{equation}
	\lim_{a\downarrow0}\left\langle\left[\Phi^{a,h}(f+g)\right]^2\right\rangle_{a,h}=\mathbb{E}\left[ \Phi^h(f+g)\right]^2 , \lim_{a\downarrow0}\left\langle\left[\Phi^{a,h}(f)\right]^2\right\rangle_{a,h}=\mathbb{E}\left[\Phi^h(f)\right]^2,
\end{equation}
which completes the proof.
\end{proof}

\begin{proof}[Proof of Proposition \ref{prop:Cov}]
We first consider the easy case: $s\neq t$. Without loss of generality, we may assume that $s<t$. An application of \eqref{eq:Sch} and \eqref{eq:Gri} gives that, for each fixed $\epsilon$ satisfying $2a\leq \epsilon<(t-s)/2$,
\begin{align}\label{eq:upperb}
&\sum_{k\in a\mathbb{Z}\cap [-L,L]}\langle \sigma_{(s_a,0)};\sigma_{(t_a,k)}\rangle_{a,h}\nonumber\\
&\qquad\leq\frac{1}{\left(\lfloor \epsilon/a\rfloor-1\right)^3}\sum_{z, w\in a\mathbb{Z}^2}1_{\{z\in [s,s+\epsilon]\times[-\epsilon/2,\epsilon/2], w\in[t-\epsilon,t]\times [-L-\epsilon,L+\epsilon]\}}\langle \sigma_z;\sigma_w\rangle_{a,h}.
\end{align}
where $\lfloor\cdot\rfloor$ denotes the greatest integer function. It is clear that
\begin{equation}\label{eq:eovera}
\lim_{a\downarrow 0}\frac{\left(\epsilon/a\right)^3}{\left(\lfloor \epsilon/a\rfloor-1\right)^3}=1.
\end{equation}
Lemma \ref{lem:momconv}, \eqref{eq:upperb} and \eqref{eq:eovera} imply that
\begin{align}\label{eq:upperb1}
&\limsup_{a\downarrow0}a^{3/4}\sum_{k\in a\mathbb{Z}\cap [-L,L]}\langle \sigma_{(s_a,0)}; \sigma_{(t_a,k)} \rangle_{a,h} \nonumber\\
&\qquad\leq\frac{1}{\epsilon^3}\int_{[s,s+\epsilon]\times[-\epsilon/2,\epsilon/2]}\int_{[t-\epsilon,t]\times [-L-\epsilon,L+\epsilon]}\hat{H}(|z-w|) dz dw \, .
\end{align}
Note that $\hat{H}$ is real analytic on $(0,\infty)$ (see, e.g., Corollary 19.5.6 of~\cite{GJ87}). Therefore, by letting $\epsilon\downarrow 0$ in \eqref{eq:upperb1}, we get
\begin{equation}\label{eq:upperb2}
\limsup_{a\downarrow0}a^{3/4}\sum_{k\in a\mathbb{Z}\cap [-L,L]}\langle \sigma_{(s_a,0)}; \sigma_{(t_a,k)} \rangle_{a,h}\leq\int_{-L}^L\hat{H}(\sqrt{(t-s)^2+y^2})dy.
\end{equation}

Another application of \eqref{eq:Sch} and \eqref{eq:Gri} gives that, for each fixed $\epsilon$ with $2a\leq\epsilon<L$,
\begin{align}\label{eq:lowerb}
&\sum_{k\in a\mathbb{Z}\cap [-L,L]}\langle \sigma_{(s_a,0)};\sigma_{(t_a,k)}\rangle_{a,h}\nonumber\\
&\qquad\geq\frac{1}{\left(\lfloor \epsilon/a\rfloor+1\right)^3}\sum_{z, w\in a\mathbb{Z}^2}1_{\{z\in [s-\epsilon,s]\times[-\epsilon/2,\epsilon/2], w\in[t,t+\epsilon]\times [-L+\epsilon,L-\epsilon]\}}\langle \sigma_z;\sigma_w\rangle_{a,h}.
\end{align}
The same arguments leading to \eqref{eq:upperb2} imply that
\begin{equation}\label{eq:lowerb2}
\liminf_{a\downarrow0}a^{3/4}\sum_{k\in a\mathbb{Z}\cap [-L,L]}\langle \sigma_{(s_a,0)}; \sigma_{(t_a,k)} \rangle_{a,h}\geq\int_{-L}^L\hat{H}(\sqrt{(t-s)^2+y^2})dy.
\end{equation}
The limit \eqref{eq:Covlimit} with $s<t$ follows from \eqref{eq:upperb2} and \eqref{eq:lowerb2}. 

The proof of \eqref{eq:ttlimit} is similar but simpler than that of \eqref{eq:Covlimit} with $s\neq t$. We next consider the more involved case:  \eqref{eq:Covlimit} with $s=t$. By the SMM inequality \eqref{eq:Sch}, we have
\begin{equation}\label{eq:ssminterm}
	\sum_{k\in a\mathbb{Z}\cap[-L,L]}\langle\sigma_{(0,0)};\sigma_{(0,k)}\rangle_{a,h}\geq \sum_{k\in a\mathbb{Z}\cap[-L,L]}\langle\sigma_{(s_a,0)};\sigma_{(t_a,k)}\rangle_{a,h} \text{ for any } s<t.
\end{equation}
Using \eqref{eq:ssminterm} and \eqref{eq:Covlimit} with $s \neq t$, we obtain
\begin{align}
	\liminf_{a\downarrow 0} a^{3/4}\sum_{k\in a\mathbb{Z}\cap[-L,L]}\langle\sigma_{(0,0)};\sigma_{(0,k)}\rangle_{a,h}&\geq \liminf_{a\downarrow 0}a^{3/4}\sum_{k\in a\mathbb{Z}\cap[-L,L]}\langle\sigma_{(s_a,0)};\sigma_{(t_a,k)}\rangle_{a,h} \\
	&=\int_{-L}^{L} H(t-s,y)dy \text{ for any } s<t.
\end{align}
Taking $t\downarrow s$, we get
\begin{equation}
	\liminf_{a\downarrow 0}a^{3/4}\sum_{k\in a\mathbb{Z}\cap[-L,L]}\langle\sigma_{(0,0)};\sigma_{(0,k)}\rangle_{a,h}\geq \int_{-L}^{L} H(0,y)dy.
\end{equation}
So \eqref{eq:Covlimit} with $s=t$ would follow if we can show
\begin{equation}\label{eq:sumgoal}
	\limsup_{a\downarrow0} a^{3/4}\sum_{k\in a\mathbb{Z}\cap[0,L]}\langle\sigma_{(0,0)};\sigma_{(0,k)}\rangle_{a,h}\leq \int_0^L H(0,y)dy.
\end{equation}

To do this, we first observe that, for any $\eta\in(0,L)$,
\begin{align}
		\limsup_{a\downarrow0} a^{3/4}\sum_{k\in a\mathbb{Z}\cap[0,L]}\langle\sigma_{(0,0)};\sigma_{(0,k)}\rangle_{a,h}\leq 	&\limsup_{a\downarrow0} a^{3/4}\sum_{k\in a\mathbb{Z}\cap[0,\eta]}\langle\sigma_{(0,0)};\sigma_{(0,k)}\rangle_{a,h}+\label{eq:sum0toeta}\\	&\limsup_{a\downarrow0} a^{3/4}\sum_{k\in a\mathbb{Z}\cap[\eta,L]}\langle\sigma_{(0,0)};\sigma_{(0,k)}\rangle_{a,h} \, . \label{eq:sumetatoL}
\end{align}
It is not hard to see that for any $\epsilon$ satisfying $2a\leq\epsilon<\eta/2$
\begin{align}\label{eq:sumcovb}
	&\sum_{k\in a\mathbb{Z}\cap[\eta,L]}\langle\sigma_{(0,0)};\sigma_{(0,k)}\rangle_{a,h}\nonumber\\
	&\qquad\leq \frac{1}{(\lfloor \epsilon/a\rfloor-1)^3}\sum_{z,w\in a\mathbb{Z}^2}1_{\{z\in[-\epsilon/2,\epsilon/2]\times[0,\epsilon],w\in[-\epsilon/2,\epsilon/2]\times[\eta-\epsilon,L+\epsilon]\}}\langle\sigma_z;\sigma_w\rangle_{a,h};
\end{align}
this is because for any $z\in a\mathbb{Z}^2\cap\{[-\epsilon/2,\epsilon/2]\times[0,\epsilon]\}$ and any $w_1\in a\mathbb{Z}\cap[-\epsilon/2,\epsilon/2]$, by the SMM inequality, we have
\begin{equation}\label{eq:sumscomp}
	\sum_{k\in a\mathbb{Z}\cap[\eta-\epsilon,L+\epsilon]}\langle \sigma_{z};\sigma_{(w_1,k)}\rangle_{a,h}\geq\sum_{k\in a\mathbb{Z}\cap[\eta,L]}\langle \sigma_{(0,0)};\sigma_{(0,k)}\rangle_{a,h}.
\end{equation}
Let us check \eqref{eq:sumscomp} for the special case $z=(-\epsilon/2,0)$ and $w_1=\epsilon/2$. 
(We assume $\epsilon/2$ is a multiple of $a$; otherwise we may take 
$\lfloor\epsilon/(2a)\rfloor a$ instead.) Part~(b) of the SMM inequality (reflection with 
respect to the line  with slope $1$ and passing through $(\epsilon/2,0)$) implies that 
\begin{equation}
		\sum_{k\in a\mathbb{Z}\cap[\eta-\epsilon,L+\epsilon]}\langle \sigma_{(-\epsilon/2,0)};\sigma_{(\epsilon/2,k)}\rangle_{a,h}\geq\sum_{k\in a\mathbb{Z}\cap[\eta-\epsilon,L-\epsilon]}\langle \sigma_{(\epsilon/2,-\epsilon)};\sigma_{(\epsilon/2,k)}\rangle_{a,h}.
\end{equation}
Then translation invariance implies
\begin{align}
		\sum_{k\in a\mathbb{Z}\cap[\eta-\epsilon,L+\epsilon]}\langle \sigma_{(-\epsilon/2,0)};\sigma_{(\epsilon/2,k)}\rangle_{a,h}&\geq 	\sum_{k\in a\mathbb{Z}\cap[\eta-\epsilon,L+\epsilon]}\langle \sigma_{(0,0)};\sigma_{(0,k+\epsilon)}\rangle_{a,h}\\
		&\geq\sum_{k\in a\mathbb{Z}\cap[\eta,L]}\langle \sigma_{(0,0)};\sigma_{(0,k)}\rangle_{a,h}.
\end{align}
Letting $a\downarrow 0$ in \eqref{eq:sumcovb} and applying Lemma \ref{lem:momconv}, we get
\begin{align}
	&\limsup_{a\downarrow0} a^{3/4}\sum_{k\in a\mathbb{Z}\cap[\eta,L]}\langle\sigma_{(0,0)};\sigma_{(0,k)}\rangle_{a,h}\nonumber\\
	&\qquad\leq\frac{1}{\epsilon^3}\int_{[-\epsilon/2,\epsilon/2]\times[0,\epsilon]}\int_{[-\epsilon/2,\epsilon/2]\times[\eta-\epsilon,L+\epsilon]}\hat{H}(|z-w|)dzdw.
\end{align}
Letting $\epsilon\downarrow 0$ in the last displayed inequality, we have
\begin{equation}\label{eq:sumetatoLbd}
	\limsup_{a\downarrow0}a^{3/4}\sum_{k\in a\mathbb{Z}\cap[\eta,L]}\langle\sigma_{(0,0)};\sigma_{(0,k)}\rangle_{a,h}\leq\int_{\eta}^L H(0,y)dy.
\end{equation}
Next, we deal with the sum in the RHS of \eqref{eq:sum0toeta}. By the GHS inequality \cite{GHS70},
\begin{equation}\label{eq:sumGHS}
	\sum_{k\in a\mathbb{Z}\cap[0,\eta]}\langle\sigma_{(0,0)};\sigma_{(0,k)}\rangle_{a,h}\leq \sum_{k\in a\mathbb{Z}\cap[0,\eta]}\langle\sigma_{(0,0)}\sigma_{(0,k)}\rangle_{a,h=0},
\end{equation}
where $\langle \cdot\rangle_{a,h=0}$ is expectation with respect to the critical 
Ising model on $a\mathbb{Z}^2$ without external field (i.e., $h=0$). Wu's 
result~\cite{Wu66, MW73} says that there exists $C_3\in(0,\infty)$ such that
\begin{equation}
	\lim_{N\rightarrow\infty}\frac{\langle\sigma_{(0,0)}\sigma_{(0,N)}\rangle_{a=1,h=0}}{N^{-1/4}}=C_3.
\end{equation}
This implies that, for all small enough $a$,
\begin{equation}
	\langle\sigma_{(0,0)}\sigma_{(0,k)}\rangle_{a,h=0}\leq 2C_3\left(\frac{k}{a}\right)^{-1/4} \text{ for any }k\in a\mathbb{Z}\cap[a^{1/2},\infty).
\end{equation}
Combining this with \eqref{eq:sumGHS}, we get
\begin{align}
	\sum_{k\in a\mathbb{Z}\cap[0,\eta]}\langle\sigma_{(0,0)};\sigma_{(0,k)}\rangle_{a,h}&\leq \sum_{k\in a\mathbb{Z}\cap[0,a^{1/2}]}1+\sum_{k\in a\mathbb{Z}\cap[a^{1/2},\eta]}\langle\sigma_{(0,0)};\sigma_{(0,k)}\rangle_{a,h=0}\\
	&\leq\frac{a^{1/2}}{a}+1+2C_3\sum_{k\in a\mathbb{Z}\cap[a^{1/2},\eta]}\left(\frac{k}{a}\right)^{-1/4}.
\end{align}
Multiplying each side by $a^{3/4}$ and taking limits, we obtain
\begin{equation}\label{eq:sum0toetabd}
	\limsup_{a\downarrow 0} a^{3/4}\sum_{k\in a\mathbb{Z}\cap[0,\eta]}\langle\sigma_{(0,0)};\sigma_{(0,k)}\rangle_{a,h}\leq 2C_3\int_0^\eta y^{-1/4} dy=\frac{8C_3}{3}\eta^{3/4}.
\end{equation}
Substituting \eqref{eq:sum0toetabd} and \eqref{eq:sumetatoLbd} into \eqref{eq:sum0toeta} and \eqref{eq:sumetatoL} respectively, we get that for all small $\eta$
\begin{equation}
		\limsup_{a\downarrow 0} a^{3/4}\sum_{k\in a\mathbb{Z}\cap[0,L]}\langle\sigma_{(0,0)};\sigma_{(0,k)}\rangle_{a,h}\leq \frac{8C_3}{3}\eta^{3/4}+\int_{\eta}^L H(0,y)dy.
\end{equation}
This completes the proof of \eqref{eq:sumgoal} by letting $\eta\downarrow 0$.
\end{proof}

\begin{remark}\label{rem:Cov}
For any $h>0$, from \eqref{eq:ttlimit} and Theorem 1.1 of \cite{CJN20a}, we know that $\hat{H}(t)$ decays exponentially for large $t$. Using this fact, it is easy to see that Proposition~\ref{prop:Cov} also extends to $L=\infty$.
\end{remark}

The following corollary, based on Proposition~\ref{prop:Cov} and Remark~\ref{rem:Cov}, 
will be used to prove convergence of $\text{Cov}(X_{L(a)}(s), X_{L(a)}(t))$.
\begin{corollary}\label{cor:Cov}
Fix $h>0$. Suppose $L(a)>0$ is a function of $a$ satisfying $L(a)\rightarrow\infty$ as $a\downarrow0$. Then for any $s,t\in\mathbb{R}$, we have
\begin{equation}
\lim_{a\downarrow 0} a^{3/4}\sum_{k\in a\mathbb{Z}\cap [-L(a),L(a)]}\langle \sigma_{(s_a,0)}; \sigma_{(t_a,k)}\rangle_{a,h}=\int_{-\infty}^{\infty} H(t-s,y) dy=K(t-s).
\end{equation}
\end{corollary}
\begin{proof}[Proof of Corollary \ref{cor:Cov}]
The GKS inequality \eqref{eq:Gri} and Remark \ref{rem:Cov} imply that
\begin{align}
\limsup_{a\downarrow0}a^{3/4}\sum_{k\in a\mathbb{Z}\cap [-L(a),L(a)]}\langle \sigma_{(s_a,0)}; \sigma_{(t_a,k)}\rangle_{a,h}&\leq\limsup_{a\downarrow0}a^{3/4}\sum_{k\in a\mathbb{Z}}\langle \sigma_{(s_a,0)}; \sigma_{(t_a,k)}\rangle_{a,h}\nonumber\\
&=\int_{-\infty}^{\infty} H(t-s,y) dy.\label{eq:Covlimsup}
\end{align}
On the other hand, for any fixed $M>0$, one may choose $\epsilon>0$ such that $L(a)\geq M$ for each $a\in(0,\epsilon)$. Then the GKS inequality \eqref{eq:Gri} and Proposition \ref{prop:Cov} imply that
\begin{align}
\liminf_{a\downarrow0}a^{3/4}\sum_{k\in a\mathbb{Z}\cap [-L(a),L(a)]}\langle \sigma_{(s_a,0)}; \sigma_{(t_a,k)}\rangle_{a,h}&\geq\liminf_{a\downarrow0}a^{3/4}\sum_{k\in a\mathbb{Z}\cap [-M,M]}\langle \sigma_{(s_a,0)}; \sigma_{(t_a,k)}\rangle_{a,h}\nonumber\\
&=\int_{-M}^{M} H(t-s,y) dy.
\end{align}
Letting $M\rightarrow\infty$, we obtain
\begin{equation}\label{eq:Covliminf}
\liminf_{a\downarrow0}a^{3/4}\sum_{k\in a\mathbb{Z}\cap [-L(a),L(a)]}\langle \sigma_{(s_a,0)}; \sigma_{(t_a,k)}\rangle_{a,h}\geq \int_{-\infty}^{\infty} H(t-s,y) dy.
\end{equation}
The corollary follows from \eqref{eq:Covlimsup} and \eqref{eq:Covliminf}.
\end{proof}

We are ready to prove convergence of $\text{Cov}(X_{L(a)}(s), X_{L(a)}(t))$ as $a\downarrow0$.
\begin{proposition}\label{prop:Cov1}
Fix $h>0$. Suppose $L(a)>0$ is a function of $a$ satisfying $L(a)\rightarrow\infty$ as $a\downarrow0$. Then for any $s,t\in \mathbb{R}$, we have
\begin{equation}
\lim_{a\downarrow0}\emph{Cov}(X_{L(a)}(s), X_{L(a)}(t))=K(t-s).
\end{equation}
\end{proposition}
\begin{proof}[Proof of Proposition \ref{prop:Cov1}]
By the definition of $X_{L(a)}(s)$ in \eqref{eq:DefX_L}, the SMM inequality \eqref{eq:Sch}, and Corollary \ref{cor:Cov}, we have
\begin{align}\label{eq:Covlimsup1}
&\limsup_{a\downarrow0}\text{Cov}(X_{L(a)}(s), X_{L(a)}(t))=\limsup_{a\downarrow0}\frac{a^{7/4}}{2L(a)}\sum_{k,j\in a\mathbb{Z}\cap [-L(a),L(a)]}\langle \sigma_{(s_a,k)};\sigma_{(t_a,j)}\rangle_{a,h}\nonumber\\
&\qquad\leq\limsup_{a\downarrow0}\frac{a^{7/4}}{2L(a)}\left(\left\lfloor\frac{2L(a)}{a}\right\rfloor+1\right)\sum_{j\in a\mathbb{Z}\cap [-L(a),L(a)]}\langle \sigma_{(s_a,0)};\sigma_{(t_a,j)}\rangle_{a,h}\nonumber\\
&\qquad=K(t-s).
\end{align}
For any fixed $\epsilon\in(0,1)$, the GKS inequality \eqref{eq:Gri}, translation invariance of the Ising model on $a\mathbb{Z}^2$ and Corollary \ref{cor:Cov} imply that
\begin{align}
&\liminf_{a\downarrow0}\text{Cov}(X_{L(a)}(s), X_{L(a)}(t))\nonumber\\
&\qquad\geq\liminf_{a\downarrow0}\frac{a^{7/4}}{2L(a)}\left\lfloor\frac{(2-2\epsilon)L(a)}{a}\right\rfloor\sum_{j\in a\mathbb{Z}\cap [-\epsilon L(a),\epsilon L(a)]}\langle \sigma_{(s_a,0)};\sigma_{(t_a,j)}\rangle_{a,h}\nonumber\\
&\qquad=(1-\epsilon)K(t-s).
\end{align}
Letting $\epsilon\downarrow0$, we get
\begin{equation}\label{eq:Covliminf1}
\liminf_{a\downarrow0}\text{Cov}(X_{L(a)}(s), X_{L(a)}(t))\geq K(t-s).
\end{equation}
The proposition now follows from \eqref{eq:Covlimsup1} and \eqref{eq:Covliminf1}.
\end{proof}

Another important ingredient in the proof of Theorem \ref{thm:main} is the following inequality from \cite{New80}.
\begin{theorem}[Theorem 1 in \cite{New80}]\label{thm:New}
Suppose $U_1,\dots,U_m$ have finite variance and satisfy the FKG inequality; then for any $r_1,\dots,r_m\in\mathbb{R}$,
\begin{equation}
\left|\left\langle \exp{\left(i\sum_{l=1}^m r_lU_l\right)}\right\rangle-\prod_{l=1}^m\langle\exp{(ir_lU_l)}\rangle\right|\leq \frac{1}{2}\mathop{\sum\sum}_{l\neq n}|r_lr_n|\emph{Cov}(U_l,U_n),
\end{equation}
where $\langle\cdot\rangle$ denotes expectation.
\end{theorem}

\begin{proof}[Proof of Theorem \ref{thm:main}]
For fixed $\vec{z}=(z_1,\dots,z_n)\in\mathbb{R}^n$ and $\vec{s}=(s_1,\dots,s_n)\in\mathbb{R}^n$, we define
\begin{equation}
Y_L=Y_L(\vec{z},\vec{s})=z_1X_L(s_1)+\cdots+z_nX_L(s_n).
\end{equation}
Note that \eqref{eq:thmmain} is equivalent to
\begin{equation}\label{eq:cha}
\lim_{a\downarrow0}\left\langle \exp(iY_{L(a)})\right\rangle_{a,h}=\exp\left(-\frac{1}{2}\sum_{j=1}^n\sum_{l=1}^n z_j z_l K(s_j-s_l)\right),
\end{equation}
for each $(z_1,\dots,z_n)\in\mathbb{R}^n$.
We will prove \eqref{eq:cha} for $(z_1,\dots,z_n)\in[0,\infty)^n$.  We claim that 
this yields a proof of Theorem \ref{thm:main} by applying the same argument used in the proof of Theorem 1 of \cite{CJN20c}; which for completeness, we reproduce below. The basic idea is to define
\begin{equation}
	Y_{L(a)}^+:=\sum_{j:z_j\geq 0}z_j X_{L(a)}(s_j), Y_{L(a)}^-:=\sum_{j:z_j< 0}|z_j| X_{L(a)}(s_j),
\end{equation}
and note that $cY_{L(a)}^++dY_{L(a)}^-$ for any $c\geq0,d\geq 0$ converges (as $a\downarrow0$) in distribution to a Gaussian distribution with mean zero and variance
\begin{equation}\label{eq:Zvar}
	\sum_{j=1}^n\sum_{l=1}^n\left(c1_{\{z_j\geq 0\}}+d1_{\{z_j<0\}}\right)\left(c1_{\{z_l\geq 0\}}+d1_{\{z_l<0\}}\right)|z_jz_l|K(s_j-s_l).
\end{equation}
So we may assume $(Y_{L(a)}^+, Y_{L(a)}^-)$ has a subsequential limit in distribution $(Y^+,Y^-)$. Then $cY^++dY^-$ for $c,d\geq 0$ is a mean zero Gaussian random variable with variance given by \eqref{eq:Zvar}. Theorem~3 of \cite{HT75} then implies that $(Y^+,Y^-)$ is a bivariate normal vector whose distribution is already determined by $cY^++dY^-$ for $c,d\geq0$. Therefore,
\begin{equation}
	(Y_{L(a)}^+, Y_{L(a)}^-)\Rightarrow (Y^+,Y^-) \text{ as }a\downarrow 0.
\end{equation}
In particular, we have
\begin{equation}
	Y_{L(a)}=Y_{L(a)}^+ - Y_{L(a)}^-\Rightarrow Y^+-Y^- \text{ as }a\downarrow 0,
\end{equation}
which is the claim.

For $j\in \mathbb{N}$ and $j\in\left[-\lfloor L^{1/2}\rfloor,\lfloor L^{1/2}\rfloor\right]$, we define
\begin{equation}
	Y_L^{(j)}:=\begin{cases}\displaystyle \sum_{l=1}^{n} z_lL^{-1/4}a^{7/8}\sum_{k\in a\mathbb{Z}\cap \left[j\lfloor L^{1/2}\rfloor,(j+1)\lfloor L^{1/2}\rfloor\right)}\left[\sigma_{(s_{l,a},k)}-\langle\sigma_{(s_{l,a},k)}\rangle_{a,h}\right],~j\leq\lfloor L^{1/2}\rfloor-1\\
		\displaystyle 2^{1/2}L^{1/4}Y_L-\sum_{k=-\lfloor L^{1/2}\rfloor}^{\lfloor L^{1/2}\rfloor-1}Y_{L}^{(k)},~j=\lfloor L^{1/2}\rfloor.
		\end{cases}
\end{equation}
where $s_{l,a}$ is a point in $a\mathbb{Z}$ which is closest to $s_l$. Note that by translation invariance, $Y_L^{(j)}$ for all $j$'s with $-\lfloor L^{1/2}\rfloor\leq j\leq\lfloor L^{1/2}\rfloor-1$ are identically distributed random variables. 
Then Theorem \ref{thm:New} implies that (this is where we use the assumption that all $z_l$'s are nonnegative)
\begin{align}
&\left|\left\langle\exp(iY_{L(a)})\right\rangle_{a,h}-\prod_{j=-\lfloor L(a)^{1/2}\rfloor}^{\lfloor L(a)^{1/2}\rfloor}\left\langle\exp(i2^{-1/2}L(a)^{-1/4}Y_{L(a)}^{(j)})\right\rangle_{a,h}\right|\nonumber\\
&\qquad\leq\frac{1}{2}\mathop{\sum\sum}_{j\neq l} 2^{-1}L(a)^{-1/2}\text{Cov}\left(Y_{L(a)}^{(j)}, Y_{L(a)}^{(l)}\right)\nonumber\\
&\qquad=\frac{1}{2}\left[\text{Var}(Y_{L(a)})-2^{-1}L(a)^{-1/2}\sum_{j=-\lfloor L(a)^{1/2}\rfloor}^{\lfloor L(a)^{1/2}\rfloor}\text{Cov}\left(Y_{L(a)}^{(j)}, Y_{L(a)}^{(j)}\right)\right]\nonumber\\
&\qquad=\frac{1}{2}\left[\text{Var}(Y_{L(a)})-\frac{\lfloor L(a)^{1/2}\rfloor}{L(a)^{1/2}}\text{Var}(Y_{L(a)}^{(0)})-\frac{1}{2\sqrt{L(a)}}\text{Var}(Y_{L(a)}^{(\lfloor L(a)^{1/2}\rfloor)})\right],
\end{align}
where we have used that $Y_{L(a)}^{(j)}$ has the same distribution as $Y_{L(a)}^{(0)}$ if $j\leq\lfloor L(a)^{1/2}\rfloor-1$. By Proposition \ref{prop:Cov1}, we have
\begin{align}
&\lim_{a\downarrow0}\text{Var}(Y_{L(a)})=\sum_{j=1}^n\sum_{l=1}^n z_jz_lK(s_j-s_l),\\
&\lim_{a\downarrow0}\text{Var}(Y_{L(a)}^{(0)})=\sum_{j=1}^n\sum_{l=1}^n z_jz_lK(s_j-s_l).\label{eq:CovY_L0}
\end{align}
By \eqref{eq:Gri}, translation invariance and the Cauchy-Schwarz inequality, we have
\begin{equation}\label{eq:cha0}
\text{Var}(Y_{L(a)}^{(\lfloor L(a)^{1/2}\rfloor)})\leq 4\text{Var}(Y_{L(a)}^{(0)}).
\end{equation}
Therefore,
\begin{equation}\label{eq:cha1}
\lim_{a\downarrow0}\left|\left\langle\exp(iY_{L(a)})\right\rangle_{a,h}-\prod_{j=-\lfloor L(a)^{1/2}\rfloor}^{\lfloor L(a)^{1/2}\rfloor}\left\langle\exp(i2^{-1/2}L(a)^{-1/4}Y_{L(a)}^{(j)})\right\rangle_{a,h}\right|=0.
\end{equation}
Since \eqref{eq:CovY_L0} implies that $\text{Var}(Y_{L(a)}^{(0)})<\infty$ for all small $a$, a standard Taylor expansion result (see, e.g., Theorem 3.3.20 of \cite{Dur19}) implies that
\begin{align}\label{eq:cha2}
&\lim_{a\downarrow0}\left[\left\langle\exp(i2^{-1/2}L(a)^{-1/4}Y_{L(a)}^{(0)})\right\rangle_{a,h}\right]^{2\lfloor L(a)^{1/2}\rfloor}\nonumber\\
&\qquad=\lim_{a\downarrow0}\left[1-4^{-1}L(a)^{-1/2}\text{Var}(Y_{L(a)}^{(0)})+o(L(a)^{-1/2})\right]^{2\lfloor L(a)^{1/2}\rfloor}\nonumber\\
&\qquad=\exp\left(-\frac{1}{2}\sum_{j=1}^n\sum_{l=1}^n z_jz_lK(s_j-s_l)\right) \, ,
\end{align}
where we have used \eqref{eq:CovY_L0} and the fact that $\lim_{x\rightarrow\infty}(1-f(x)/x)^x=e^{-A}$ if $\lim_{x\rightarrow\infty}f(x)=A\in\mathbb{R}$. Now it is easy to see that \eqref{eq:cha} follows from \eqref{eq:cha0}, \eqref{eq:cha1} and \eqref{eq:cha2}.
\end{proof}

We conclude with a proof of Theorem~\ref{thm:gend}.
\begin{proof}[Proof of Theorem \ref{thm:gend}]
Without loss of generality, we may assume $a=1$ in the proof. Theorem \ref{thm:gend} is equivalent to
\begin{equation}\label{eq:cha21}
\lim_{L\rightarrow\infty}\left\langle \exp(i(z_1\hat{X}^a_L(s_1)+\cdots+ z_n\hat{X}^a_L(s_n)))\right\rangle=\exp\left(-\frac{1}{2}\sum_{j=1}^n\sum_{l=1}^n z_j z_l \hat{K}^a(s_j-s_l)\right),
\end{equation}
for each $(z_1,\dots,z_n)\in\mathbb{R}^n$.
By the same argument as in the proof of Theorem \ref{thm:main}, it is enough to show \eqref{eq:cha21} for $(z_1,\dots,z_n)\in[0,\infty)^n$. Note that $\langle\sigma_z;\sigma_w\rangle$ decays exponentially as $|z-w|\rightarrow\infty$ (see \cite{LP68,FR17}). This and (26) in Lemma 4 of \cite{New80} imply that
\begin{equation}\label{eq:cha22}
\lim_{L\rightarrow\infty}\frac{\text{Var}(\sum_{\vec{y}\in \mathbb{Z}^{d-1}\cap[-L,L]^{d-1}}\sigma_{(s,\vec{y})})}{(2L)^{d-1}}=\sum_{\vec{y}\in\mathbb{Z}^{d-1}}\langle \sigma_{(0,\vec{0})};\sigma_{(0,\vec{y})}\rangle,~\forall s\in \mathbb{Z}.
\end{equation}
A direct application of Theorem 2 in \cite{New80} implies that
\begin{align}
\frac{\sum_{j=1}^n z_j\sum_{\vec{y}\in \mathbb{Z}^{d-1}\cap[-L,L]^{d-1}}\sigma_{(s_j,\vec{y})}-\sum_{j=1}^n z_j\sum_{\vec{y}\in \mathbb{Z}^{d-1}\cap[-L,L]^{d-1}}\langle\sigma_{(s_j,\vec{y})}\rangle}{(2L)^{(d-1)/2}}
\end{align}
converges as $L\uparrow\infty$ in distribution to a normal distribution with mean $0$ and variance
\begin{equation}
\sum_{j=1}^n\sum_{l=1}^n z_j z_l\sum_{y\in\mathbb{Z}^{d-1}}\langle\sigma_{(s_j,\vec{0})};\sigma_{(s_l,\vec{y})}\rangle \, .
\end{equation}
This combined with \eqref{eq:cha22} completes the proof of \eqref{eq:cha21}.
\end{proof}

\section*{Acknowledgements}
The research of the second author was partially supported by NSFC grant 11901394 and STCSM grant 17YF1413300, and that of the third author by US-NSF grant DMS-1507019. 

\bibliographystyle{abbrv}
\bibliography{reference}

\begin{thebibliography}{10}

\bibitem{BG11}
D.~Borthwick and S.~Garibaldi.
\newblock Did a 1-dimensional magnet detect a 248-dimensional {Lie} algebra?
\newblock {\em Notices of the AMS}, 58(8):1055--1066, 2011.

\bibitem{CGN15}
F.~Camia, C.~Garban, and C.~M. Newman.
\newblock Planar {Ising} magnetization field {I. Uniqueness} of the critical
  scaling limit.
\newblock {\em The Annals of Probability}, 43(2):528--571, 2015.

\bibitem{CGN16}
F.~Camia, C.~Garban, and C.~M. Newman.
\newblock Planar {Ising} magnetization field {II. Properties} of the critical
  and near-critical scaling limits.
\newblock {\em Annales de l'IHP, Probabilit{\'e}s et Statistiques},
  52(1):146--161, 2016.

\bibitem{CJN20a}
F.~Camia, J.~Jiang, and C.~M. Newman.
\newblock Exponential decay for the near-critical scaling limit of the planar
  {Ising} model.
\newblock {\em Communications on Pure and Applied Mathematics},
  73(7):1371--1405, 2020.

\bibitem{CJN20b}
F.~Camia, J.~Jiang, and C.~M. Newman.
\newblock {FK--Ising} coupling applied to near-critical planar models.
\newblock {\em Stochastic Processes and their Applications}, 130(2):560--583,
  2020.

\bibitem{CJN20c}
F.~Camia, J.~Jiang, and C.~M. Newman.
\newblock A {Gaussian} process related to the mass spectrum of the
  near-critical {Ising} model.
\newblock {\em Journal of Statistical Physics}, 179:885--900, 2020.

\bibitem{CJN20d}
F.~Camia, J.~Jiang, and C.~M. Newman.
\newblock Conformal measure ensembles and planar {Ising} magnetization: a
  review.
\newblock {\em To appear in Markov Processes and Related Fields}, 2021.

\bibitem{CJN21}
F.~Camia, J.~Jiang, and C.~M. Newman.
\newblock The effect of free boundary conditions on the {Ising} model in high
  dimensions.
\newblock {\em To appear in Probability Theory and Related Fields}, 2021.

\bibitem{CH00}
M.~Caselle and M.~Hasenbusch.
\newblock Critical amplitudes and mass spectrum of the 2d {Ising} model in a
  magnetic field.
\newblock {\em Nuclear Physics B}, 579(3):667--703, 2000.

\bibitem{Che18}
D.~Chelkak.
\newblock 2{D Ising} model: Correlation functions at criticality via
  {Riemann}-type boundary value problems.
\newblock In {\em European Congress of Mathematics}, pages 235--256, 2018.

\bibitem{CHI15}
D.~Chelkak, C.~Hongler, and K.~Izyurov.
\newblock Conformal invariance of spin correlations in the planar {Ising}
  model.
\newblock {\em Annals of Mathematics}, pages 1087--1138, 2015.

\bibitem{CTW+10}
R.~Coldea, D.~A. Tennant, E.~M. Wheeler, E.~Wawrzynska, D.~Prabhakaran,
  M.~Telling, K.~Habicht, P.~Smeibidl, and K.~Kiefer.
\newblock Quantum criticality in an {Ising} chain: experimental evidence for
  emergent $e_8$ symmetry.
\newblock {\em Science}, 327(5962):177--180, 2010.

\bibitem{Del04}
G.~Delfino.
\newblock Integrable field theory and critical phenomena: the {Ising} model in
  a magnetic field.
\newblock {\em Journal of Physics A: Mathematical and General}, 37(14):R45,
  2004.

\bibitem{Dur19}
R.~Durrett.
\newblock {\em Probability: theory and examples}, volume~49.
\newblock Cambridge university press, 2019.

\bibitem{FR17}
J.~Fr{\"o}hlich and P.-F. Rodr{\'\i}guez.
\newblock On cluster properties of classical ferromagnets in an external
  magnetic field.
\newblock {\em Journal of Statistical Physics}, 166(3-4):828--840, 2017.

\bibitem{FM17}
M.~Furlan and J.-C. Mourrat.
\newblock A tightness criterion for random fields, with application to the
  {Ising} model.
\newblock {\em Electronic Journal of Probability}, 22:1--29, 2017.

\bibitem{GJ87}
J.~Glimm and A.~Jaffe.
\newblock {\em Quantum physics: a functional integral point of view}.
\newblock Springer, 1987.

\bibitem{Gri67}
R.~B. Griffiths.
\newblock Correlations in {Ising} ferromagnets. {I}.
\newblock {\em Journal of Mathematical Physics}, 8(3):478--483, 1967.

\bibitem{Gri67b}
R.~B. Griffiths.
\newblock Correlations in {Ising} ferromagnets. {II. External} magnetic fields.
\newblock {\em Journal of Mathematical Physics}, 8(3):484--489, 1967.

\bibitem{GHS70}
R.~B. Griffiths, C.~A. Hurst, and S.~Sherman.
\newblock Concavity of magnetization of an {Ising} ferromagnet in a positive
  external field.
\newblock {\em Journal of Mathematical Physics}, 11(3):790--795, 1970.

\bibitem{HT75}
G.~G. Hamedani and M.~N. Tata.
\newblock On the determination of the bivariate normal distribution from
  distributions of linear combinations of the variables.
\newblock {\em The American Mathematical Monthly}, 82(9):913--915, 1975.

\bibitem{KS68}
D.~G. Kelly and S.~Sherman.
\newblock General {Griffiths'} inequalities on correlations in {Ising}
  ferromagnets.
\newblock {\em Journal of Mathematical Physics}, 9(3):466--484, 1968.

\bibitem{KR21}
F.~R. Klausen and A.~Raoufi.
\newblock Mass scaling of the near-critical 2d {Ising} model using random
  currents.
\newblock {\em arXiv preprint arXiv:2105.13673}, 2021.

\bibitem{LP68}
J.~L. Lebowitz and O.~Penrose.
\newblock Analytic and clustering properties of thermodynamic functions and
  distribution functions for classical lattice and continuum systems.
\newblock {\em Communications in Mathematical Physics}, 11(2):99--124, 1968.

\bibitem{LY52}
T.-D. Lee and C.-N. Yang.
\newblock Statistical theory of equations of state and phase transitions. {II}.
  {Lattice} gas and {Ising} model.
\newblock {\em Physical Review}, 87(3):410, 1952.

\bibitem{MM12}
B.~M. McCoy and J.-M. Maillard.
\newblock The importance of the {Ising} model.
\newblock {\em Progress of Theoretical Physics}, 127(5):791--817, 2012.

\bibitem{MW73}
B.~M. McCoy and T.~T. Wu.
\newblock {\em The two-dimensional {Ising} model}.
\newblock Harvard University Press, 1973.

\bibitem{MMS77}
A.~Messager and S.~Miracle-Sole.
\newblock Correlation functions and boundary conditions in the {Ising}
  ferromagnet.
\newblock {\em Journal of Statistical Physics}, 17(4):245--262, 1977.

\bibitem{MM97}
I.~Montvay and G.~M{\"u}nster.
\newblock {\em Quantum fields on a lattice}.
\newblock Cambridge University Press, 1997.

\bibitem{New80}
C.~M. Newman.
\newblock Normal fluctuations and the {FKG} inequalities.
\newblock {\em Communications in Mathematical Physics}, 74(2):119--128, 1980.

\bibitem{OS73}
K.~Osterwalder and R.~Schrader.
\newblock Axioms for {Euclidean Green's} functions.
\newblock {\em Communications in Mathematical Physics}, 31(2):83--112, 1973.

\bibitem{OS75}
K.~Osterwalder and R.~Schrader.
\newblock Axioms for {Euclidean Green's functions II}.
\newblock {\em Communications in Mathematical Physics}, 42(3):281--305, 1975.

\bibitem{Ott20b}
S.~Ott.
\newblock Sharp asymptotics for the truncated two-point function of the {Ising}
  model with a positive field.
\newblock {\em Communications in Mathematical Physics}, 374(3):1361--1387,
  2020.

\bibitem{Sch77}
R.~Schrader.
\newblock New correlation inequalities for the {Ising} model and ${P}(\varphi)$
  theories.
\newblock {\em Physical Review B}, 15(5):2798, 1977.

\bibitem{Wu66}
T.~T. Wu.
\newblock Theory of {Toeplitz} determinants and the spin correlations of the
  two-dimensional {Ising} model. {I}.
\newblock {\em Physical Review}, 149(1):380, 1966.

\bibitem{Zam89b}
A.~B. Zamolodchikov.
\newblock Integrable field theory from conformal field theory.
\newblock In {\em Integrable Systems in Quantum Field Theory and Statistical
  Mechanics}, pages 641--674. Mathematical Society of Japan, 1989.

\bibitem{Zam89a}
A.~B. Zamolodchikov.
\newblock Integrals of motion and {S}-matrix of the (scaled) $t= t_c$ {Ising}
  model with magnetic field.
\newblock {\em International Journal of Modern Physics A}, 4(16):4235--4248,
  1989.

\end{thebibliography}
\end{document}